\documentclass[a4paper,12pt]{article}
\usepackage[]{graphicx}
\usepackage[]{amsmath,amssymb}
\usepackage[]{amsthm,enumerate}
\usepackage{verbatim,bm}
\usepackage{a4wide}

\usepackage{color}
\usepackage{xcolor}
\usepackage{url}
\makeatletter
\def\bbordermatrix#1{\begingroup \m@th
  \@tempdima 4.75\p@
  \setbox\z@\vbox{%
    \def\cr{\crcr\noalign{\kern2\p@\global\let\cr\endline}}%
    \ialign{$##$\hfil\kern2\p@\kern\@tempdima&\thinspace\hfil$##$\hfil
      &&\quad\hfil$##$\hfil\crcr
      \omit\strut\hfil\crcr\noalign{\kern-\baselineskip}%
      #1\crcr\omit\strut\cr}}%
  \setbox\tw@\vbox{\unvcopy\z@\global\setbox\@ne\lastbox}%
  \setbox\tw@\hbox{\unhbox\@ne\unskip\global\setbox\@ne\lastbox}%
  \setbox\tw@\hbox{$\kern\wd\@ne\kern-\@tempdima\left[\kern-\wd\@ne
    \global\setbox\@ne\vbox{\box\@ne\kern2\p@}%
    \vcenter{\kern-\ht\@ne\unvbox\z@\kern-\baselineskip}\,\right]$}%
  \null\;\vbox{\kern\ht\@ne\box\tw@}\endgroup}
\makeatother

\newtheorem{theorem}{Theorem}[section]
\newtheorem{lemma}[theorem]{Lemma}
\newtheorem{corol}[theorem]{Corollary}
\theoremstyle{definition}
\newtheorem{definition}[theorem]{Definition}
\newtheorem{example}[theorem]{Example}

\theoremstyle{remark}
\newtheorem{remark}[theorem]{Remark}

\newcommand{\defn}[1]{{\em #1}}

\def\0{{\bm 0}}   

\title{Embeddable partial Hadamard matrices related to the projective planes}
\author{
 Hadi Kharaghani\thanks{Department of Mathematics and Computer Science, University of Lethbridge,
Lethbridge, Alberta, T1K 3M4, Canada. \texttt{kharaghani@uleth.ca}}
\and
  Sho Suda\thanks{Department of Mathematics,  National Defense Academy of Japan, Yokosuka, Kanagawa 239-8686, Japan. \texttt{ssuda@nda.ac.jp}}\and
Yash Shamsundar Khobragade\thanks{Indian Institute of Science Education and Research (IISER) Bhopal, India. \texttt{yashk19@iiserb.ac.in}}
}
\date{\today}

\begin{document}
\maketitle

\begin{abstract}
The existence of a projective plane of order $p\equiv3\pmod{4}$, where $p$ is a prime power, is shown to be equivalent to the existence of a balancedly multi-splittable embeddable $p^2\times p(p+1)$ partial Hadamard matrix. 

\end{abstract}

\section{Introduction}
Balancedly splittable Hadamard (BSH) matrices were introduced by Kharaghani-Suda in \cite{split-h}. 
The definition is broad in scope in \cite{split-h}. Still, in this paper, we will consider the following definition: a Hadamard matrix of order $n^2$ is balancedly splittable if it is Hadamard equivalent to the block form $\begin{bmatrix}H_1 & H_2\end{bmatrix}$, where $H_1$ is a $n^2\times \frac{1}{2}(n^2-n)$ submatrix, and all the off-diagonal entries of $H_1H_1^\top$ are of the absolute values $\frac{n}{2}$.
The plentiful of such matrices was shown in an extension to balancedly splittable orthogonal designs (BSOD) by Kharaghani-Pender-Suda in \cite{split-od} every BSH matrix of order $n^2$ leads to two optimal sets of equiangular lines and unbiased bases. In a recent paper \cite{KS}, Kharaghani-Suda introduced a class of balancedly multi-splittable Hadamard matrices (BMSH) and demonstrated that their existence is equivalent to projective planes of doubly even order. Such a property limits the existence of BMSH, possibly only to those of order $4^m$. 

In this work, the focus is on the prime powers $p\equiv3\pmod{4}$. Plenty of balancedly multi-splittable partial Hadamard (BMSPH) matrices exist for the prime powers. Such matrices are row and column regular and embeddable in a Hadamard matrix. These features will be used to introduce a class of balanced multi-splittable BIBD (BMSBIBD). The paper's main result is a complete characterization of BMSPH matrices and BMSBIBD in terms of orthogonal arrays and projective planes.

\section{Balancedly multi-splittable Hadamard matrices for $p\equiv3\pmod{4}$}
Let $p$ be a positive integer such that $p\equiv3\pmod{4}$.   
Let $H$ be a $p^2\times p(p+1)$ partial Hadamard matrix. 
A $p^2\times p(p+1)$ partial Hadamard matrix $H$ is said to be \defn{balancedly multi-splittable} if there is a block form of $H=\begin{bmatrix}
H_1 & \cdots & H_{p+1}  
\end{bmatrix}$ 
such that each $H_i$ is a $p^2\times p$ matrix and $H$ is balancedly spllitable with respect to a submatrix $\begin{bmatrix}
 H_{i_1} & \cdots & H_{i_{(p+1)/2}}  
\end{bmatrix}$ for any $(p+1)/2$-elements subset $\{i_1,\ldots,i_{(p+1)/2}\}$ of $\{1,2,\ldots,p+1\}$.

\begin{theorem}\label{thm:main}
Let $p$ be a positive integer such that $p\equiv3\pmod{4}$.    
The following are equivalent. 
\begin{enumerate}
\item There exists a balancedly multi-splittable $p^2\times p(p+1)$ partial Hadamard matrix. 
\item There exist an OA$_1(p^2,p+1,p,2)$ and a $p\times(p+1)$ partial Hadamard matrix.  
\end{enumerate}
\end{theorem}

\begin{theorem}\label{thm:main2}
Let $p$ be a positive integer such that $p\equiv3\pmod{4}$.    
Let $H$ be a balancedly multi-splittable $p^2\times p(p+1)$ partial Hadamard matrix. 
\begin{enumerate}
\item There exist signed permutation matrices $P$ of order $p^2$ and $Q$ of order $p(p+1)$ such that $H':=PHQ$ is row-regular and column-regular. 
\item $H'$ is embeddable in a Hadamard matrix of order $p(p+1)$.    
\end{enumerate}
\end{theorem}

\subsection{Proof of Theorem~\ref{thm:main}}
The proof from (ii) to (i) is the same as \cite{KS}. 

\begin{proof}[Proof of (i) $\Rightarrow$ (ii)]

Assume that $H$ is a balancedly multi-splittable $p^2\times p(p+1)$ partial Hadamard matrix with respect to the following block form: 
$$
H=\begin{bmatrix} H_1 & \cdots & H_{p+1}  
\end{bmatrix},
$$ 
where each $H_i$ is a $p^2\times p$ matrix. 


Consider two rows indexed by $x$ and $y$. 
Write the $x$-th row of $H_i$ as $r_{x,i}$ and similar to $y$. 
Define $a_i=\langle r_{x,i}, r_{y,i}\rangle$ for $i\in\{1,\ldots,p+1\}$. 
We then show that $a_i\in \{p,-1\}$ for any $i$ or $a_i\in \{-p,1\}$ for any $i$.  

By the definition of multi-splittable, for any $I\in \binom{[p+1]}{(p+1)/2}$
\begin{align}\label{eq:p1}
\sum_{i\in I}a_i\in\{\pm (p+1)/2\},  
\end{align}
where $[p+1]=\{1,\ldots,p+1\}$ and $\binom{[p+1]}{(p+1)/2}$ is the set of $(p+1)/2$-elements  subsets of $[p+1]$. 
We formulate these equations as a matrix equation as follows. 
Define a $\binom{p+1}{(p+1)/2}\times (p+1)$ matrix $W$ with rows indexed by the elements of $\binom{[p+1]}{(p+1)/2}$ and the columns indexed by the elements of $[p+1]$ as 
\[
W_{I,i}=\begin{cases}
1 & \text{ if } i\in I,\\
0 & \text{ otherwise},
\end{cases}
\]
and set a column vector ${\bm a}=(a_i)_{i\in [p+1]}$. 
Then the equations \eqref{eq:p1} read as
\begin{align}\label{eq:p2}
W{\bm a}=(p+1)\chi-\frac{p+1}{2}{\bm 1}
\end{align}
for some $(0,1)$-column vector $\chi$, where ${\bm 1}$ is the all-ones column vector. 
Letting ${\bm b}=\frac{1}{p+1}({\bm a}+{\bm 1})$, the equation \eqref{eq:p2} is 
\begin{align}\label{eq:p3}
W{\bm b}=\chi. 
\end{align}
Let $C$ be a subset of $\binom{[p+1]}{(p+1)/2}$ whose characteristic vector is $\chi$. 
Since the column vectors of $W$ are linearly independent,  it is easy to see that the following hold: 
\begin{itemize}
\item $\chi={\bm 0}$ if and only if ${\bm b}={\bm 0}$, 
\item $\chi={\bm 1}$ if and only if ${\bm b}=\frac{2}{p+1}{\bm 1}$. 
\end{itemize}
From now, we assume that $\chi$ is neither ${\bm 0}$ nor ${\bm 1}$. 
Then:

{\bf Claim 1}: For any $I\in \binom{[p+1]}{(p+1)/2}$, $\chi(I)\neq \chi(\binom{[p+1]}{(p+1)/2}\setminus I)$. 
\begin{proof}[Proof of Claim 1]
Without loss of generality, we may assume that $I=\{1,\ldots,(p+1)/2\}$. 
Assume to the contrary that $\chi(I)= \chi(\binom{[p+1]}{(p+1)/2}\setminus I)$. 
First we deal with the case $\chi(I)=0$ and thus $\chi(\binom{[p+1]}{(p+1)/2}\setminus I)=0$. 
Then 
\[
0=\chi(I)=\sum_{i=1}^{(p+1)/2}b_i, 
\]
where ${\bm b}=(b_i)_{i=1}^{p+1}$. 
For any $i\in \{1,\ldots,(p+1)/2\}$, consider $I'=I\cup\{i+(p+1)/2\}\setminus \{i\}$. 
Then 
\[
\chi(I')=\chi(I)-b_i+b_{i+(p+1)/2}.
\]
Since $\chi(I')\in\{0,1\}$, $b_{i+(p+1)/2}$ equals either $b_i$ or $b_i+1$. 
Set $b_{i+(p+1)/2}=b_i+\varepsilon_i$. 
Then 
\begin{align*}
0&=\chi\left(\binom{[p+1]}{(p+1)/2}\setminus I\right)\\
&=\sum_{i=1}^{(p+1)/2}b_{i+(p+1)/2}\\
&=\sum_{i=1}^{(p+1)/2}(b_i+\varepsilon_i)\\
&=\chi(I)+\sum_{i=1}^{(p+1)/2}\varepsilon_i\\
&=\sum_{i=1}^{(p+1)/2}\varepsilon_i.  
\end{align*}
Since each $\varepsilon_i \in\{0,1\}$, $\varepsilon_i=0$ and $b_{i+(p+1)/2}=b_i$ hold for each $i$. 
Doing the above argument for any bijection $\varphi:\{1,\ldots,(p+1)/2\}\rightarrow \{1+(p+1)/2,\ldots,p+1\}$, we obtain $b_{\varphi(i)}=b_i$ for any $i$. Therefore there exists a constant $c$ such that $b_i=c$ for any $i\in \{1,\ldots,p+1\}$. 
Since $\chi(I)=0$, we obtain $c=0$ and thus $\chi={\bm 0}$, a contradiction. 
The case of $\chi(I)=1$ is similar. 
\end{proof}
 By Claim 1, the width of the subset $C$ is at most $w:=(p+1)/2-1$ in the vertex set of the Johnson scheme $J(p+1,(p+1)/2)$ \cite{BGKM}. 
 On the other hand, since $W{\bm b}=\chi$, $E_i\chi={\bm 0}$ for $j\in\{2,\ldots,(p+1)/2\}$ where $E_i$ is the primitive idempotent of the Johnson scheme. 
 Therefore, the dual width of $C$ is $w^*:=1$. 
 By the inequality $w+w^*\geq (p+1)/2$, it holds that $w+w^*=(p+1)/2$. 
 
 Then it follows from \cite[Theorem~8]{BGKM} that $C$ is either 
 \begin{itemize}
 \item $C=\{\alpha\in \binom{[p+1]}{(p+1)/2} \mid T\subset \alpha\}$ for some $T\subset [p+1]$ such that $|T|=1$, or 
 \item $C=\{\alpha\in \binom{[p+1]}{(p+1)/2} \mid  \alpha\subset T \}$ for some $T\subset [p+1]$ such that $|T|=(p+1)/2-1$. 
 \end{itemize}
 Since the column vectors of $W$ are linearly independent, the solution of the equation \eqref{eq:p3} for each $\chi$ above are 
 \begin{itemize}
 \item ${\bm b}=e_i$, where $e_i$ is the $i$-th standard basis, 
 \item ${\bm b}=\frac{2}{p+1}{\bm 1}-e_i$. 
 \end{itemize}
 Then the corresponding ${\bm a}$'s are  
 \begin{itemize}
 \item ${\bm a}=(p+1)e_i-{\bm 1}$, 
 \item ${\bm a}=(p+1)(\frac{2}{p+1}{\bm 1}-e_i)-{\bm 1}={\bm 1}-(p+1)e_i$. 
 \end{itemize} 
Therefore, 
\begin{align}\label{eq:p4}
\langle r_{x,i}, r_{y,i}\rangle\in\{p,-1\}\text{ for any $i$,} \text{ or } \langle r_{x,i}, r_{y,i}\rangle\in\{-p,1\}\text{ for any $i$}. 
\end{align}
That is,  $H_iH_i^\top$ is a $(p,-p,1,-1)$-matrix for any $i$.  

\medskip
{\bf Claim 2} There exists a diagonal matrix $D$ with diagonal entries in $\{1,-1\}$ such that $DH_1(DH_1)^\top$ is a $(p,-1)$-matrix.

Note that this result yields that $DH_i(DH_i)^\top$ is a $(p,-1)$-matrix for any $i\in\{2,\ldots,p+1\}$ by \eqref{eq:p4}. 
\begin{proof}[Proof of Claim 2]
Define the equivalence relation $\sim_1$ on $[p^2]$ by 
\[
x \sim_1 y \text{ if and only if }r_{x,1}=\pm r_{y,1}. 
\]
Let $C_1,\ldots,C_\ell$ be the equivalence classes with respect to $\sim_1$ on $[p^2]$. 
Then we may negate some of the elements $r_{z,1}$ in $z\in [p^2]$ so that if $x$ and $y$ are in the same equivalence class, then $r_{x,1}=r_{y,1}$.  

Assume that $1\in C_1$ and fix it. 
Now we negate some of the elements $r_{y,1}$ for $y\in [p^2]\setminus C_1$ so that $\langle r_{1,1},r_{y,1} \rangle=-1$. 
Note that the condition that $r_{x,1}=r_{y,1}$ for any $x$ and $y$ being in the same class keeps. 
Then, the matrix $H_1H_1^\top$ forms like: 
\begin{align*}
H_1H_1^\top=\begin{bmatrix}
pJ_{|C_1|} & -J_{|C_1|,|C_2|} & \cdots & -J_{|C_1|,|C_\ell|} \\
-J_{|C_2|,|C_1|} & pJ_{|C_1|} &  \cdots & \pm J_{|C_2|,|C_\ell|} \\
\vdots & \vdots &  \ddots & \vdots \\
-J_{|C_\ell|,|C_1|} & \pm J_{|C_\ell|,|C_2|} & \cdots & pJ_{|C_\ell|} 
\end{bmatrix},
\end{align*}
where $J_k$ is the $k\times k$ all-ones matrix and $J_{k,\ell}$ is the $k\times \ell$ all-ones matrix.  
We then show the $(i,j)$-block for any distinct $i,j\in\{2,\ldots,\ell\}$ is $-J_{|C_i|,|C_j|}$. 
Assume that there exist $r_{x,1}$ and $r_{y,1}$ with the property that $\langle r_{x,1},r_{y,1} \rangle=1$. 
We may assume that $r_{1,1}={\bm 1}^\top$ by negating columns of $H_1$.
Write the entries of $r_{1,1},r_{x,1},r_{y,1}$ as: 
\begin{equation*}
\begin{array}{lccccc}
r_{1,1}=(&+ \cdots + & + \cdots + & + \cdots + & + \cdots + &),\\
r_{x,1}=(&+ \cdots + & + \cdots + & - \cdots - & - \cdots - &),\\
r_{y,1}=(&
\underbrace{+ \cdots +}_{a \text{ columns}} &
\underbrace{- \cdots -}_{b \text{ columns}} &
\underbrace{+ \cdots +}_{c \text{ columns}} &
\underbrace{- \cdots -}_{d \text{ columns}}&).
\end{array}
\end{equation*}
Then by the information of inner products among $r_{1,1},r_{x,1},r_{y,1}$, 
\begin{align*}
a+b+c+d&=p \\
a+b-c-d&=-1 \\
a-b+c-d&=-1 \\
a-b-c+d&=1.
\end{align*}
The solution of the system of linear equations is 
\[
a=b=c=\frac{p-1}{4},\quad d=\frac{p+3}{4}.
\]
This contradicts the fact that $p\equiv 3\pmod{4}$. 
Therefore $\langle r_{x,i},r_{y,i}\rangle=-1$, and thus  
\begin{align*}
H_1H_1^\top=\begin{bmatrix}
pJ_{|C_1|} & -J_{|C_1|,|C_2|} & \cdots & -J_{|C_1|,|C_\ell|} \\
-J_{|C_2|,|C_1|} & pJ_{|C_1|} &  \cdots &  -J_{|C_2|,|C_\ell|} \\
\vdots & \vdots &  \ddots & \vdots \\
-J_{|C_\ell|,|C_1|} & -J_{|C_\ell|,|C_2|} & \cdots & pJ_{|C_\ell|} 
\end{bmatrix}. 
\end{align*}
\end{proof}
The condition \eqref{eq:p4} is now 
\begin{align}\label{eq:p5}
\langle r_{x,i}, r_{y,i}\rangle\in\{p,-1\}\text{ for any $i$}. 
\end{align}

Next, we show the following:

\medskip
{\bf Claim 3} For any $i$, by suitably permuting the rows of $H$,  
\begin{align*}
H_iH_i^\top=\begin{bmatrix}
pJ_{p} & -J_{p,p} & \cdots & -J_{p,p} \\
-J_{p,p} & pJ_{p} &  \cdots &  -J_{p,p} \\
\vdots & \vdots &  \ddots & \vdots \\
-J_{p,p} & -J_{p,p} & \cdots & pJ_{p} 
\end{bmatrix}. 
 \end{align*}  

\begin{proof}[Proof of Claim 3]
We show it for $i=1$. 
By \eqref{eq:p4}, 
\begin{align*}
H_k H_k^\top=\begin{bmatrix}
(p+1)I_{|C_1|}-J_{|C_1|} & * & \cdots & * \\
* & (p+1)I_{|C_2|}-J_{|C_2|} &  \cdots &  * \\
\vdots & \vdots &  \ddots & \vdots \\
* & * & \cdots & (p+1)I_{|C_\ell|}-J_{|C_\ell|} 
\end{bmatrix},  
\end{align*}
for $k\in\{2,\ldots,p+1\}$. 

Let $B_{i,j}^{k}$ be the $(i,j)$-block of $H_kH_k^\top$. 
Let $i,j$ be distinct integers in $\{1,\ldots,\ell\}$. 
For distinct $x,y\in C_i$ and $z\in C_j$, consider the $3\times 3$ principal submatrix, say $M$, of $H_kH_k^\top$ obtained by restricting the rows and the columns to $\{x,y,z\}$. 
Since $M$ is positive semidefinite, the matrix $M$ cannot be
\[
\bbordermatrix{~ & x & y & z \cr
              x & p & -1 & p \cr
              y & -1 & p & p \cr
              z & p & p & p \cr}.
\]
Thus, each row and each column of $B_{i,j}^k$ has at most one entry of $p$.  
Since $\sum_{i=1}^{p+1}H_iH_i^\top=p(p+1)I$, each $C_i$ must have the same size $p$ and $\ell$ must be $p$.  
\end{proof}

We are ready to prove the theorem. 
For each $i$, consider the set of the distinct $p$ rows of $H_i$ and fix the ordering of the row vectors, say $\tilde{r}_{1,i},\ldots,\tilde{r}_{p,i}$ 
Consider the matrix 
\[
K_i=\begin{bmatrix}
1 & \tilde{r}_{1,i} \\ 
1 & \tilde{r}_{2,i} \\ 
\vdots & \vdots \\ 
1 & \tilde{r}_{p,i} 
\end{bmatrix}. 
\]
Then  each $K_i$ is a $p\times (p+1)$  partial Hadamard matrix.  
Replace $\tilde{r}_{j,i}$ in $H$ with a symbol $j$, and we obtain $p^2\times (p+1)$ matrix, say $C$. 
Then, the code consisting of the row vectors of $C$ is a  $p$-ary code of length $p+1$ with the number of codewords $p^2$, equidistance $p$, of length $p+1$.   
Indeed, it follows from \eqref{eq:p5} that the Hamming distance between any distinct codewords is $p$.  
Therefore, there exists an orthogonal array $\text{OA}_1(p^2,p+1,p,2)$. 
This completes the proof.  
\end{proof}




\subsection{Proof of Theorem~\ref{thm:main2}}
\begin{proof}[Proof of Theorem~\ref{thm:main2} (i)]
Assume that $H$ is a balancedly multi-splittable $p^2\times p(p+1)$ partial Hadamard matrix with respect to the following block form: 
\[
H=\begin{bmatrix} H_1 & \cdots & H_{p+1}  
\end{bmatrix},
\] 
where each $H_i$ is a $p^2\times p$ matrix.  
With reference to proof of Theorem~\ref{thm:main}, 
consider the following $p\times p$ matrix
\[
L_i=\begin{bmatrix}
 \tilde{r}_{1,i} \\ 
 \tilde{r}_{2,i} \\ 
 \vdots \\ 
 \tilde{r}_{p,i} 
\end{bmatrix}.  
\]
Then $L_iL_i^\top=(p+1)I_p-J_p$ holds. 
Consider the sum of the row vectors of $L_i$, say ${\bm x}_i=\sum_{j=1}^p \tilde{r}_{j,i}$. 
Since $p$ is an odd integer, the entries of ${\bm x}_i$ are odd integers. 
Moreover, 
\begin{align*}
\langle {\bm x}_i, {\bm x}_i\rangle &=\sum_{j,k=1}^p \langle \tilde{r}_{j,i},\tilde{r}_{k,i}\rangle \\
&=\sum_{j=1}^p \langle \tilde{r}_{j,i},\tilde{r}_{j,i}\rangle+\sum_{j\neq k} \langle \tilde{r}_{j,i},\tilde{r}_{k,i}\rangle \\
&=\sum_{j=1}^p p+\sum_{j\neq k} (-1) \\
&=p^2+(-1)p(p-1)\\
&=p.
\end{align*}
Therefore, ${\bm x}_i$ must be a $(1,-1)$-vector. 
Then the matrix 
\[
\tilde{L}_i=\begin{bmatrix}
1 & {\bm x}_i \\
-1 & \tilde{r}_{1,i} \\ 
-1 &  \tilde{r}_{2,i} \\ 
\vdots  &  \vdots \\ 
-1 &  \tilde{r}_{p,i} 
\end{bmatrix}  
\]
is easily shown to be a Hadamard matrix of order $p+1$. 
Define the diagonal matrix.
\[
D_i=\text{diag}(x_{1,i},\ldots,x_{p,i}), 
\]
where ${\bm x}_i=(x_{1,i},\ldots,x_{p,i})$. 
Then, the matrix $H_iD_i$ has the constant row sum $1$ and the constant column sum $p$. 
Therefore, the matrix $H'$ defined by 
\[
H'=\begin{bmatrix} H_1D_1 & \cdots & H_{p+1}D_{p+1}  
\end{bmatrix}
\]
is the desired matrix. 
This completes the proof.  
\end{proof}

\begin{proof}[Proof of Theorem~\ref{thm:main2} (ii)]
In proof of Theorem~\ref{thm:main2} (i), it was shown that there exist  Hadamard matrices $\tilde{L}_i$ of order $p+1$ for each $i\in\{1,\ldots,p+1\}$. 
Fix one of them and normalize so that the first row consists of all ones, say $L$ with rows $r_1=(1,\ldots,1),r_2,\ldots,r_{p+1}$.
Define $M_i=r_i^\top {\bm 1}$ for any $i\in\{1,\ldots,p+1\}$.  
Then, it is easy to see that the matrix. 
\[
\begin{bmatrix} H_1D_1 & \cdots & H_{p+1}D_{p+1}  \\
M_1 & \cdots & M_{p+1}  
\end{bmatrix}
\]
is a Hadamard matrix of order $p(p+1)$. 
\end{proof}

\section{Balancedly multi-splittable Hadamard matrices for $q\equiv1\pmod{4}$}
Let $q$ be a positive integer such that $q\equiv 1\pmod{4}$.   
Let $H$ be a $q^2\times(q+1)$ partial quaternary Hadamard matrix. 
A $q^2\times(q+1)$ partial Hadamard matrix $H$ is said to be \defn{balancedly multi-splittable} if there is a block form of $H=\begin{bmatrix}
H_1 & \cdots & H_{p+1}  
\end{bmatrix}$ 
such that each $H_i$ is a $q^2\times q$ matrix and $H$ is balancedly spllitable with respect to a submatrix $H'=\begin{bmatrix}
 H_{i_1} & \cdots & H_{i_{(q+1)/2}}  
\end{bmatrix}$ for any $(q+1)/2$-elements subset $\{i_1,\ldots,i_{(q+1)/2}\}$ of $\{1,2,\ldots,q+1\}$, that is 
\[
(H'(H')^*)_{xy}=\begin{cases}
\frac{q+1}{2} & \text{ if }x=y,\\
\pm q & \text{ if }x\neq y. 
\end{cases}
\]

In the same manner as the implication from (ii) to (i) in Theorem~\ref{thm:main}, the following is proved.  
\begin{theorem}\label{thm:mainq}
Let $q$ be a positive integer such that $q\equiv1\pmod{4}$. In the following, (ii) implies (i). 
\begin{enumerate}
\item There exists a balancedly multi-splittable $q^2\times q(q+1)$ partial quaternary Hadamard matrix. 
\item There exist an OA$_1(q^2,q+1,q,2)$ and a $q\times(q+1)$ partial quaternary Hadamard matrix.  
\end{enumerate}
\end{theorem}

\medskip
\noindent{\bf Question~1.} 
Let $q$ be a positive integer such that $q\equiv1\pmod{4}$. In Theorem~\ref{thm:mainq}, does (i) imply (ii)?

\section{Balancedly splittable BIBDs}
For the prime power $p\equiv3\pmod{4}$, referring to Theorem \ref{thm:main2}, it is assumed that any balancedly multi-splittable partial Hadamard matrix $H=\begin{bmatrix}
H_1 & \cdots & H_{p+1}  
\end{bmatrix}$  of order $p^2\times p(p+1)$ is both row and column regular, and the row sum of each of $H_i$ is one, and the column sum is $p$.  Furthermore, let $R_i=\begin{bmatrix}r_1, & \ldots & r_{p+1}\end{bmatrix}$ and $R_j=\begin{bmatrix}s_1, & \ldots & s_{p+1}\end{bmatrix}$, $i\ne j$ be any two rows of $H$, where $r_i$ and $s_i$ denote the block rows corresponding to $H_i$. It follows from Theorem \ref{thm:main} that exactly one of the $r_i$'s and $s_i$'s occur in the same column.
Let $D=1/2(J_{p^2\times p(p+1)} - H)$,  then it follows (as shown below) from Lemma \ref{equal2} that the $(0,1)$-matrix $D$ is the incidence matrix of a BIBD$(p^2,p^2+p,1/2(p^2-1),1/2(p^2-p),1/4(p^2-p-2)$. Furthermore, $D$ is partitioned into $p+1$ sets of blocks $D_i$, $i=1,2,\cdots,p+1$,  each consisting of $p$ columns in such a way that for any selection of $D_i$, $i= 1, 2, \cdots,1/2(p+1)$ each pair of distinct points occurs in either $1/8(p^2-1)$ or  $1/8(p^2-2p-3)$ blocks. This motivates the following definition. The main reference for terms not defined here is \cite{stinson}.

\begin{definition}
Let $p\equiv3\pmod{4}$.   
Let $D$ be the incidence matrix of a BIBD, say $\mathcal {D}$ with parameters $$\left(p^2,p^2+p,1/2(p^2-1),1/2(p^2-p),1/4(p^2-p-2)\right).$$
$\mathcal{D}$ is said to be \defn{balancedly multi-splittable} if there is a block form of $D=\begin{bmatrix}
D_1 & \cdots & D_{p+1}  
\end{bmatrix}$ 
such that each $D_i$ is a $p^2\times p$ matrix with row sum $\frac{p-1}{2}$, and for any submatrix $\begin{bmatrix}
 D_{i_1} & \cdots & D_{i_{(p+1)/2}}  \end{bmatrix}$, for any $(p+1)/2$-elements subset $\{i_1,\ldots,i_{(p+1)/2}\}$ of $\{1,2,\ldots,p+1\}$,
each pair of distinct points occurs in either $\frac{p^2-1}{8}$ or  $\frac{p^2-2p-3}{8}$ blocks.


\end{definition}

\begin{example}
A balancedly multi-splittable BIBD$(49,56,24,21,10)$. 
\[\left[ \arraycolsep=1pt \def\arraystretch{1} 
\footnotesize  \begin{array}{ccccccc|ccccccc|ccccccc|ccccccc|ccccccc|ccccccc|ccccccc|ccccccc}
0 & 0 & 0 & 1 & 0 & 1 & 1 & 0 & 0 & 0 & 1 & 0 & 1 & 1 & 0 & 0 & 0 & 1 & 0 & 1 & 1 & 0 & 0 & 0 & 1 & 0 & 1 & 1 & 0 & 0 & 0 & 1 & 0 & 1 & 1 & 0 & 0 & 0 & 1 & 0 & 1 & 1 & 0 & 0 & 0 & 1 & 0 & 1 & 1 & 0 & 0 & 0 & 1 & 0 & 1 & 1 
\\
 0 & 0 & 0 & 1 & 0 & 1 & 1 & 1 & 0 & 0 & 0 & 1 & 0 & 1 & 1 & 0 & 0 & 0 & 1 & 0 & 1 & 1 & 0 & 0 & 0 & 1 & 0 & 1 & 1 & 0 & 0 & 0 & 1 & 0 & 1 & 1 & 0 & 0 & 0 & 1 & 0 & 1 & 1 & 0 & 0 & 0 & 1 & 0 & 1 & 1 & 0 & 0 & 0 & 1 & 0 & 1 
\\
 0 & 0 & 0 & 1 & 0 & 1 & 1 & 1 & 1 & 0 & 0 & 0 & 1 & 0 & 1 & 1 & 0 & 0 & 0 & 1 & 0 & 1 & 1 & 0 & 0 & 0 & 1 & 0 & 1 & 1 & 0 & 0 & 0 & 1 & 0 & 1 & 1 & 0 & 0 & 0 & 1 & 0 & 1 & 1 & 0 & 0 & 0 & 1 & 0 & 1 & 1 & 0 & 0 & 0 & 1 & 0 
\\
 0 & 0 & 0 & 1 & 0 & 1 & 1 & 0 & 1 & 1 & 0 & 0 & 0 & 1 & 0 & 1 & 1 & 0 & 0 & 0 & 1 & 0 & 1 & 1 & 0 & 0 & 0 & 1 & 0 & 1 & 1 & 0 & 0 & 0 & 1 & 0 & 1 & 1 & 0 & 0 & 0 & 1 & 0 & 1 & 1 & 0 & 0 & 0 & 1 & 0 & 1 & 1 & 0 & 0 & 0 & 1 
\\
 0 & 0 & 0 & 1 & 0 & 1 & 1 & 1 & 0 & 1 & 1 & 0 & 0 & 0 & 1 & 0 & 1 & 1 & 0 & 0 & 0 & 1 & 0 & 1 & 1 & 0 & 0 & 0 & 1 & 0 & 1 & 1 & 0 & 0 & 0 & 1 & 0 & 1 & 1 & 0 & 0 & 0 & 1 & 0 & 1 & 1 & 0 & 0 & 0 & 1 & 0 & 1 & 1 & 0 & 0 & 0 
\\
 0 & 0 & 0 & 1 & 0 & 1 & 1 & 0 & 1 & 0 & 1 & 1 & 0 & 0 & 0 & 1 & 0 & 1 & 1 & 0 & 0 & 0 & 1 & 0 & 1 & 1 & 0 & 0 & 0 & 1 & 0 & 1 & 1 & 0 & 0 & 0 & 1 & 0 & 1 & 1 & 0 & 0 & 0 & 1 & 0 & 1 & 1 & 0 & 0 & 0 & 1 & 0 & 1 & 1 & 0 & 0 
\\
 0 & 0 & 0 & 1 & 0 & 1 & 1 & 0 & 0 & 1 & 0 & 1 & 1 & 0 & 0 & 0 & 1 & 0 & 1 & 1 & 0 & 0 & 0 & 1 & 0 & 1 & 1 & 0 & 0 & 0 & 1 & 0 & 1 & 1 & 0 & 0 & 0 & 1 & 0 & 1 & 1 & 0 & 0 & 0 & 1 & 0 & 1 & 1 & 0 & 0 & 0 & 1 & 0 & 1 & 1 & 0 
\\
 1 & 0 & 0 & 0 & 1 & 0 & 1 & 0 & 0 & 0 & 1 & 0 & 1 & 1 & 1 & 0 & 0 & 0 & 1 & 0 & 1 & 1 & 1 & 0 & 0 & 0 & 1 & 0 & 0 & 1 & 1 & 0 & 0 & 0 & 1 & 1 & 0 & 1 & 1 & 0 & 0 & 0 & 0 & 1 & 0 & 1 & 1 & 0 & 0 & 0 & 0 & 1 & 0 & 1 & 1 & 0 
\\
 1 & 0 & 0 & 0 & 1 & 0 & 1 & 1 & 0 & 0 & 0 & 1 & 0 & 1 & 1 & 1 & 0 & 0 & 0 & 1 & 0 & 0 & 1 & 1 & 0 & 0 & 0 & 1 & 1 & 0 & 1 & 1 & 0 & 0 & 0 & 0 & 1 & 0 & 1 & 1 & 0 & 0 & 0 & 0 & 1 & 0 & 1 & 1 & 0 & 0 & 0 & 0 & 1 & 0 & 1 & 1 
\\
 1 & 0 & 0 & 0 & 1 & 0 & 1 & 1 & 1 & 0 & 0 & 0 & 1 & 0 & 0 & 1 & 1 & 0 & 0 & 0 & 1 & 1 & 0 & 1 & 1 & 0 & 0 & 0 & 0 & 1 & 0 & 1 & 1 & 0 & 0 & 0 & 0 & 1 & 0 & 1 & 1 & 0 & 0 & 0 & 0 & 1 & 0 & 1 & 1 & 1 & 0 & 0 & 0 & 1 & 0 & 1 
\\
 1 & 0 & 0 & 0 & 1 & 0 & 1 & 0 & 1 & 1 & 0 & 0 & 0 & 1 & 1 & 0 & 1 & 1 & 0 & 0 & 0 & 0 & 1 & 0 & 1 & 1 & 0 & 0 & 0 & 0 & 1 & 0 & 1 & 1 & 0 & 0 & 0 & 0 & 1 & 0 & 1 & 1 & 1 & 0 & 0 & 0 & 1 & 0 & 1 & 1 & 1 & 0 & 0 & 0 & 1 & 0 
\\
 1 & 0 & 0 & 0 & 1 & 0 & 1 & 1 & 0 & 1 & 1 & 0 & 0 & 0 & 0 & 1 & 0 & 1 & 1 & 0 & 0 & 0 & 0 & 1 & 0 & 1 & 1 & 0 & 0 & 0 & 0 & 1 & 0 & 1 & 1 & 1 & 0 & 0 & 0 & 1 & 0 & 1 & 1 & 1 & 0 & 0 & 0 & 1 & 0 & 0 & 1 & 1 & 0 & 0 & 0 & 1 
\\
 1 & 0 & 0 & 0 & 1 & 0 & 1 & 0 & 1 & 0 & 1 & 1 & 0 & 0 & 0 & 0 & 1 & 0 & 1 & 1 & 0 & 0 & 0 & 0 & 1 & 0 & 1 & 1 & 1 & 0 & 0 & 0 & 1 & 0 & 1 & 1 & 1 & 0 & 0 & 0 & 1 & 0 & 0 & 1 & 1 & 0 & 0 & 0 & 1 & 1 & 0 & 1 & 1 & 0 & 0 & 0 
\\
 1 & 0 & 0 & 0 & 1 & 0 & 1 & 0 & 0 & 1 & 0 & 1 & 1 & 0 & 0 & 0 & 0 & 1 & 0 & 1 & 1 & 1 & 0 & 0 & 0 & 1 & 0 & 1 & 1 & 1 & 0 & 0 & 0 & 1 & 0 & 0 & 1 & 1 & 0 & 0 & 0 & 1 & 1 & 0 & 1 & 1 & 0 & 0 & 0 & 0 & 1 & 0 & 1 & 1 & 0 & 0 
\\
 1 & 1 & 0 & 0 & 0 & 1 & 0 & 0 & 0 & 0 & 1 & 0 & 1 & 1 & 1 & 1 & 0 & 0 & 0 & 1 & 0 & 1 & 0 & 1 & 1 & 0 & 0 & 0 & 0 & 0 & 1 & 0 & 1 & 1 & 0 & 1 & 0 & 0 & 0 & 1 & 0 & 1 & 0 & 1 & 1 & 0 & 0 & 0 & 1 & 0 & 1 & 0 & 1 & 1 & 0 & 0 
\\
 1 & 1 & 0 & 0 & 0 & 1 & 0 & 1 & 1 & 0 & 0 & 0 & 1 & 0 & 1 & 0 & 1 & 1 & 0 & 0 & 0 & 0 & 0 & 1 & 0 & 1 & 1 & 0 & 1 & 0 & 0 & 0 & 1 & 0 & 1 & 0 & 1 & 1 & 0 & 0 & 0 & 1 & 0 & 1 & 0 & 1 & 1 & 0 & 0 & 0 & 0 & 0 & 1 & 0 & 1 & 1 
\\
 1 & 1 & 0 & 0 & 0 & 1 & 0 & 1 & 0 & 1 & 1 & 0 & 0 & 0 & 0 & 0 & 1 & 0 & 1 & 1 & 0 & 1 & 0 & 0 & 0 & 1 & 0 & 1 & 0 & 1 & 1 & 0 & 0 & 0 & 1 & 0 & 1 & 0 & 1 & 1 & 0 & 0 & 0 & 0 & 0 & 1 & 0 & 1 & 1 & 1 & 1 & 0 & 0 & 0 & 1 & 0 
\\
 1 & 1 & 0 & 0 & 0 & 1 & 0 & 0 & 0 & 1 & 0 & 1 & 1 & 0 & 1 & 0 & 0 & 0 & 1 & 0 & 1 & 0 & 1 & 1 & 0 & 0 & 0 & 1 & 0 & 1 & 0 & 1 & 1 & 0 & 0 & 0 & 0 & 0 & 1 & 0 & 1 & 1 & 1 & 1 & 0 & 0 & 0 & 1 & 0 & 1 & 0 & 1 & 1 & 0 & 0 & 0 
\\
 1 & 1 & 0 & 0 & 0 & 1 & 0 & 1 & 0 & 0 & 0 & 1 & 0 & 1 & 0 & 1 & 1 & 0 & 0 & 0 & 1 & 0 & 1 & 0 & 1 & 1 & 0 & 0 & 0 & 0 & 0 & 1 & 0 & 1 & 1 & 1 & 1 & 0 & 0 & 0 & 1 & 0 & 1 & 0 & 1 & 1 & 0 & 0 & 0 & 0 & 0 & 1 & 0 & 1 & 1 & 0 
\\
 1 & 1 & 0 & 0 & 0 & 1 & 0 & 0 & 1 & 1 & 0 & 0 & 0 & 1 & 0 & 1 & 0 & 1 & 1 & 0 & 0 & 0 & 0 & 0 & 1 & 0 & 1 & 1 & 1 & 1 & 0 & 0 & 0 & 1 & 0 & 1 & 0 & 1 & 1 & 0 & 0 & 0 & 0 & 0 & 1 & 0 & 1 & 1 & 0 & 1 & 0 & 0 & 0 & 1 & 0 & 1 
\\
 1 & 1 & 0 & 0 & 0 & 1 & 0 & 0 & 1 & 0 & 1 & 1 & 0 & 0 & 0 & 0 & 0 & 1 & 0 & 1 & 1 & 1 & 1 & 0 & 0 & 0 & 1 & 0 & 1 & 0 & 1 & 1 & 0 & 0 & 0 & 0 & 0 & 1 & 0 & 1 & 1 & 0 & 1 & 0 & 0 & 0 & 1 & 0 & 1 & 0 & 1 & 1 & 0 & 0 & 0 & 1 
\\
 0 & 1 & 1 & 0 & 0 & 0 & 1 & 0 & 0 & 0 & 1 & 0 & 1 & 1 & 0 & 1 & 1 & 0 & 0 & 0 & 1 & 0 & 0 & 1 & 0 & 1 & 1 & 0 & 1 & 1 & 0 & 0 & 0 & 1 & 0 & 0 & 1 & 0 & 1 & 1 & 0 & 0 & 1 & 0 & 0 & 0 & 1 & 0 & 1 & 1 & 0 & 1 & 1 & 0 & 0 & 0 
\\
 0 & 1 & 1 & 0 & 0 & 0 & 1 & 0 & 1 & 1 & 0 & 0 & 0 & 1 & 0 & 0 & 1 & 0 & 1 & 1 & 0 & 1 & 1 & 0 & 0 & 0 & 1 & 0 & 0 & 1 & 0 & 1 & 1 & 0 & 0 & 1 & 0 & 0 & 0 & 1 & 0 & 1 & 1 & 0 & 1 & 1 & 0 & 0 & 0 & 0 & 0 & 0 & 1 & 0 & 1 & 1 
\\
 0 & 1 & 1 & 0 & 0 & 0 & 1 & 0 & 0 & 1 & 0 & 1 & 1 & 0 & 1 & 1 & 0 & 0 & 0 & 1 & 0 & 0 & 1 & 0 & 1 & 1 & 0 & 0 & 1 & 0 & 0 & 0 & 1 & 0 & 1 & 1 & 0 & 1 & 1 & 0 & 0 & 0 & 0 & 0 & 0 & 1 & 0 & 1 & 1 & 0 & 1 & 1 & 0 & 0 & 0 & 1 
\\
 0 & 1 & 1 & 0 & 0 & 0 & 1 & 1 & 1 & 0 & 0 & 0 & 1 & 0 & 0 & 1 & 0 & 1 & 1 & 0 & 0 & 1 & 0 & 0 & 0 & 1 & 0 & 1 & 1 & 0 & 1 & 1 & 0 & 0 & 0 & 0 & 0 & 0 & 1 & 0 & 1 & 1 & 0 & 1 & 1 & 0 & 0 & 0 & 1 & 0 & 0 & 1 & 0 & 1 & 1 & 0 
\\
 0 & 1 & 1 & 0 & 0 & 0 & 1 & 0 & 1 & 0 & 1 & 1 & 0 & 0 & 1 & 0 & 0 & 0 & 1 & 0 & 1 & 1 & 0 & 1 & 1 & 0 & 0 & 0 & 0 & 0 & 0 & 1 & 0 & 1 & 1 & 0 & 1 & 1 & 0 & 0 & 0 & 1 & 0 & 0 & 1 & 0 & 1 & 1 & 0 & 1 & 1 & 0 & 0 & 0 & 1 & 0 
\\
 0 & 1 & 1 & 0 & 0 & 0 & 1 & 1 & 0 & 0 & 0 & 1 & 0 & 1 & 1 & 0 & 1 & 1 & 0 & 0 & 0 & 0 & 0 & 0 & 1 & 0 & 1 & 1 & 0 & 1 & 1 & 0 & 0 & 0 & 1 & 0 & 0 & 1 & 0 & 1 & 1 & 0 & 1 & 1 & 0 & 0 & 0 & 1 & 0 & 0 & 1 & 0 & 1 & 1 & 0 & 0 
\\
 0 & 1 & 1 & 0 & 0 & 0 & 1 & 1 & 0 & 1 & 1 & 0 & 0 & 0 & 0 & 0 & 0 & 1 & 0 & 1 & 1 & 0 & 1 & 1 & 0 & 0 & 0 & 1 & 0 & 0 & 1 & 0 & 1 & 1 & 0 & 1 & 1 & 0 & 0 & 0 & 1 & 0 & 0 & 1 & 0 & 1 & 1 & 0 & 0 & 1 & 0 & 0 & 0 & 1 & 0 & 1 
\\
 1 & 0 & 1 & 1 & 0 & 0 & 0 & 0 & 0 & 0 & 1 & 0 & 1 & 1 & 1 & 0 & 1 & 1 & 0 & 0 & 0 & 1 & 0 & 0 & 0 & 1 & 0 & 1 & 0 & 1 & 0 & 1 & 1 & 0 & 0 & 1 & 1 & 0 & 0 & 0 & 1 & 0 & 0 & 0 & 1 & 0 & 1 & 1 & 0 & 0 & 1 & 1 & 0 & 0 & 0 & 1 
\\
 1 & 0 & 1 & 1 & 0 & 0 & 0 & 1 & 0 & 1 & 1 & 0 & 0 & 0 & 1 & 0 & 0 & 0 & 1 & 0 & 1 & 0 & 1 & 0 & 1 & 1 & 0 & 0 & 1 & 1 & 0 & 0 & 0 & 1 & 0 & 0 & 0 & 1 & 0 & 1 & 1 & 0 & 0 & 1 & 1 & 0 & 0 & 0 & 1 & 0 & 0 & 0 & 1 & 0 & 1 & 1 
\\
 1 & 0 & 1 & 1 & 0 & 0 & 0 & 1 & 0 & 0 & 0 & 1 & 0 & 1 & 0 & 1 & 0 & 1 & 1 & 0 & 0 & 1 & 1 & 0 & 0 & 0 & 1 & 0 & 0 & 0 & 1 & 0 & 1 & 1 & 0 & 0 & 1 & 1 & 0 & 0 & 0 & 1 & 0 & 0 & 0 & 1 & 0 & 1 & 1 & 1 & 0 & 1 & 1 & 0 & 0 & 0 
\\
 1 & 0 & 1 & 1 & 0 & 0 & 0 & 0 & 1 & 0 & 1 & 1 & 0 & 0 & 1 & 1 & 0 & 0 & 0 & 1 & 0 & 0 & 0 & 1 & 0 & 1 & 1 & 0 & 0 & 1 & 1 & 0 & 0 & 0 & 1 & 0 & 0 & 0 & 1 & 0 & 1 & 1 & 1 & 0 & 1 & 1 & 0 & 0 & 0 & 1 & 0 & 0 & 0 & 1 & 0 & 1 
\\
 1 & 0 & 1 & 1 & 0 & 0 & 0 & 1 & 1 & 0 & 0 & 0 & 1 & 0 & 0 & 0 & 1 & 0 & 1 & 1 & 0 & 0 & 1 & 1 & 0 & 0 & 0 & 1 & 0 & 0 & 0 & 1 & 0 & 1 & 1 & 1 & 0 & 1 & 1 & 0 & 0 & 0 & 1 & 0 & 0 & 0 & 1 & 0 & 1 & 0 & 1 & 0 & 1 & 1 & 0 & 0 
\\
 1 & 0 & 1 & 1 & 0 & 0 & 0 & 0 & 0 & 1 & 0 & 1 & 1 & 0 & 0 & 1 & 1 & 0 & 0 & 0 & 1 & 0 & 0 & 0 & 1 & 0 & 1 & 1 & 1 & 0 & 1 & 1 & 0 & 0 & 0 & 1 & 0 & 0 & 0 & 1 & 0 & 1 & 0 & 1 & 0 & 1 & 1 & 0 & 0 & 1 & 1 & 0 & 0 & 0 & 1 & 0 
\\
 1 & 0 & 1 & 1 & 0 & 0 & 0 & 0 & 1 & 1 & 0 & 0 & 0 & 1 & 0 & 0 & 0 & 1 & 0 & 1 & 1 & 1 & 0 & 1 & 1 & 0 & 0 & 0 & 1 & 0 & 0 & 0 & 1 & 0 & 1 & 0 & 1 & 0 & 1 & 1 & 0 & 0 & 1 & 1 & 0 & 0 & 0 & 1 & 0 & 0 & 0 & 1 & 0 & 1 & 1 & 0 
\\
 0 & 1 & 0 & 1 & 1 & 0 & 0 & 0 & 0 & 0 & 1 & 0 & 1 & 1 & 0 & 1 & 0 & 1 & 1 & 0 & 0 & 0 & 1 & 1 & 0 & 0 & 0 & 1 & 1 & 0 & 0 & 0 & 1 & 0 & 1 & 0 & 0 & 1 & 0 & 1 & 1 & 0 & 1 & 0 & 1 & 1 & 0 & 0 & 0 & 1 & 1 & 0 & 0 & 0 & 1 & 0 
\\
 0 & 1 & 0 & 1 & 1 & 0 & 0 & 0 & 1 & 0 & 1 & 1 & 0 & 0 & 0 & 1 & 1 & 0 & 0 & 0 & 1 & 1 & 0 & 0 & 0 & 1 & 0 & 1 & 0 & 0 & 1 & 0 & 1 & 1 & 0 & 1 & 0 & 1 & 1 & 0 & 0 & 0 & 1 & 1 & 0 & 0 & 0 & 1 & 0 & 0 & 0 & 0 & 1 & 0 & 1 & 1 
\\
 0 & 1 & 0 & 1 & 1 & 0 & 0 & 0 & 1 & 1 & 0 & 0 & 0 & 1 & 1 & 0 & 0 & 0 & 1 & 0 & 1 & 0 & 0 & 1 & 0 & 1 & 1 & 0 & 1 & 0 & 1 & 1 & 0 & 0 & 0 & 1 & 1 & 0 & 0 & 0 & 1 & 0 & 0 & 0 & 0 & 1 & 0 & 1 & 1 & 0 & 1 & 0 & 1 & 1 & 0 & 0 
\\
 0 & 1 & 0 & 1 & 1 & 0 & 0 & 1 & 0 & 0 & 0 & 1 & 0 & 1 & 0 & 0 & 1 & 0 & 1 & 1 & 0 & 1 & 0 & 1 & 1 & 0 & 0 & 0 & 1 & 1 & 0 & 0 & 0 & 1 & 0 & 0 & 0 & 0 & 1 & 0 & 1 & 1 & 0 & 1 & 0 & 1 & 1 & 0 & 0 & 0 & 1 & 1 & 0 & 0 & 0 & 1 
\\
 0 & 1 & 0 & 1 & 1 & 0 & 0 & 0 & 0 & 1 & 0 & 1 & 1 & 0 & 1 & 0 & 1 & 1 & 0 & 0 & 0 & 1 & 1 & 0 & 0 & 0 & 1 & 0 & 0 & 0 & 0 & 1 & 0 & 1 & 1 & 0 & 1 & 0 & 1 & 1 & 0 & 0 & 0 & 1 & 1 & 0 & 0 & 0 & 1 & 1 & 0 & 0 & 0 & 1 & 0 & 1 
\\
 0 & 1 & 0 & 1 & 1 & 0 & 0 & 1 & 0 & 1 & 1 & 0 & 0 & 0 & 1 & 1 & 0 & 0 & 0 & 1 & 0 & 0 & 0 & 0 & 1 & 0 & 1 & 1 & 0 & 1 & 0 & 1 & 1 & 0 & 0 & 0 & 1 & 1 & 0 & 0 & 0 & 1 & 1 & 0 & 0 & 0 & 1 & 0 & 1 & 0 & 0 & 1 & 0 & 1 & 1 & 0 
\\
 0 & 1 & 0 & 1 & 1 & 0 & 0 & 1 & 1 & 0 & 0 & 0 & 1 & 0 & 0 & 0 & 0 & 1 & 0 & 1 & 1 & 0 & 1 & 0 & 1 & 1 & 0 & 0 & 0 & 1 & 1 & 0 & 0 & 0 & 1 & 1 & 0 & 0 & 0 & 1 & 0 & 1 & 0 & 0 & 1 & 0 & 1 & 1 & 0 & 1 & 0 & 1 & 1 & 0 & 0 & 0 
\\
 0 & 0 & 1 & 0 & 1 & 1 & 0 & 0 & 0 & 0 & 1 & 0 & 1 & 1 & 0 & 0 & 1 & 0 & 1 & 1 & 0 & 0 & 1 & 0 & 1 & 1 & 0 & 0 & 1 & 0 & 1 & 1 & 0 & 0 & 0 & 0 & 1 & 1 & 0 & 0 & 0 & 1 & 1 & 1 & 0 & 0 & 0 & 1 & 0 & 1 & 0 & 0 & 0 & 1 & 0 & 1 
\\
 0 & 0 & 1 & 0 & 1 & 1 & 0 & 0 & 0 & 1 & 0 & 1 & 1 & 0 & 0 & 1 & 0 & 1 & 1 & 0 & 0 & 1 & 0 & 1 & 1 & 0 & 0 & 0 & 0 & 1 & 1 & 0 & 0 & 0 & 1 & 1 & 1 & 0 & 0 & 0 & 1 & 0 & 1 & 0 & 0 & 0 & 1 & 0 & 1 & 0 & 0 & 0 & 1 & 0 & 1 & 1 
\\
 0 & 0 & 1 & 0 & 1 & 1 & 0 & 0 & 1 & 0 & 1 & 1 & 0 & 0 & 1 & 0 & 1 & 1 & 0 & 0 & 0 & 0 & 1 & 1 & 0 & 0 & 0 & 1 & 1 & 1 & 0 & 0 & 0 & 1 & 0 & 1 & 0 & 0 & 0 & 1 & 0 & 1 & 0 & 0 & 0 & 1 & 0 & 1 & 1 & 0 & 0 & 1 & 0 & 1 & 1 & 0 
\\
 0 & 0 & 1 & 0 & 1 & 1 & 0 & 1 & 0 & 1 & 1 & 0 & 0 & 0 & 0 & 1 & 1 & 0 & 0 & 0 & 1 & 1 & 1 & 0 & 0 & 0 & 1 & 0 & 1 & 0 & 0 & 0 & 1 & 0 & 1 & 0 & 0 & 0 & 1 & 0 & 1 & 1 & 0 & 0 & 1 & 0 & 1 & 1 & 0 & 0 & 1 & 0 & 1 & 1 & 0 & 0 
\\
 0 & 0 & 1 & 0 & 1 & 1 & 0 & 0 & 1 & 1 & 0 & 0 & 0 & 1 & 1 & 1 & 0 & 0 & 0 & 1 & 0 & 1 & 0 & 0 & 0 & 1 & 0 & 1 & 0 & 0 & 0 & 1 & 0 & 1 & 1 & 0 & 0 & 1 & 0 & 1 & 1 & 0 & 0 & 1 & 0 & 1 & 1 & 0 & 0 & 1 & 0 & 1 & 1 & 0 & 0 & 0 
\\
 0 & 0 & 1 & 0 & 1 & 1 & 0 & 1 & 1 & 0 & 0 & 0 & 1 & 0 & 1 & 0 & 0 & 0 & 1 & 0 & 1 & 0 & 0 & 0 & 1 & 0 & 1 & 1 & 0 & 0 & 1 & 0 & 1 & 1 & 0 & 0 & 1 & 0 & 1 & 1 & 0 & 0 & 1 & 0 & 1 & 1 & 0 & 0 & 0 & 0 & 1 & 1 & 0 & 0 & 0 & 1 
\\
 0 & 0 & 1 & 0 & 1 & 1 & 0 & 1 & 0 & 0 & 0 & 1 & 0 & 1 & 0 & 0 & 0 & 1 & 0 & 1 & 1 & 0 & 0 & 1 & 0 & 1 & 1 & 0 & 0 & 1 & 0 & 1 & 1 & 0 & 0 & 1 & 0 & 1 & 1 & 0 & 0 & 0 & 0 & 1 & 1 & 0 & 0 & 0 & 1 & 1 & 1 & 0 & 0 & 0 & 1 & 0 
\end{array}\right]
\]
The inner product of any two rows from any set of four column blocks is in the set $\{4,6\}$. 
\end{example}

\begin{theorem}\label{equal2}
Let $p$ be a positive integer such that $p\equiv3\pmod{4}$.    
The following are equivalent. 
\begin{enumerate}
\item There exists a balancedly multi-splittable $p^2\times p(p+1)$ partial Hadamard matrix. 
\item There exists a balancedly multi-splittable BIBD with parameters $$\left(p^2,p^2+p,1/2(p^2-1),1/2(p^2-p),1/4(p^2-p-2)\right).$$
\end{enumerate}
\end{theorem}
\begin{proof}[Proof of (i) $\Rightarrow$ (ii)]
Let $H$ be a balancedly multi-splittable $p^2\times p(p+1)$ partial Hadamard matrix for the following block form: 
$$
H=\begin{bmatrix} H_1 & \cdots & H_{p+1}  
\end{bmatrix},
$$ 
where each $H_i$ is a $p^2\times p$ matrix. 
Let   $$D=1/2(J_{p^2\times p(p+1)}-H)=\begin{bmatrix} D_1 & \cdots & D_{p+1}  
\end{bmatrix}.$$
Each of $D_i$'s is a $(0,1)$-matrix, and 
$$DD^\top=\frac{1}{4}\left((P^2-p-2)J_{p^2}+(p+p^2)I_{p^2}\right).$$
 Therefore, $D$ is the incidence matrix of a  
 $$\text{BIBD}\left(p^2,p^2+p,1/2(p^2-1),1/2(p^2-p),1/4(p^2-p-2)\right).$$
 Furthermore, let  $D_{(p+1)/2}=\begin{bmatrix}D_{i_1} & \cdots & D_{i_{(p+1)/2}}\end{bmatrix}$, for any $(p+1)/2$-elements subset $\{i_1,\ldots,i_{(p+1)/2}\}$ of $\{1,2,\ldots,p+1\}$. Let $r_i$ and $r_j$ be two distinct rows in 
 $D_{(p+1)/2}$ and let $R_i$ and $R_j$ be two rows of the corresponding $H_t$'s matrices. Then 
 $r_i=1/2(e_p-R_i)$, and $r_j=1/2(e_p-R_j)$, where $e_x$ denotes the all one vector of dimension $x$. 
 Noting that 
 
\[ 
  \langle r_{i}, r_j\rangle=\begin{cases}
  \frac{p^2-1}{8} & \text{if} \langle R_i,R_i\rangle=1/2(p+1),\\
 \frac{p^2-2p-3}{8} & \text{if} \langle R_i,R_i\rangle=-1/2(p+1),
 \end{cases}
 \]
which shows that $D$ is balancedly multi-splittable. 
\end{proof}

\begin{proof}[Proof of (ii) $\Rightarrow$ (i)]

The proof essentially reverses the process in the proof of (i) $\Rightarrow$ (ii).
Let $D=\begin{bmatrix} D_1 & \cdots & D_{p+1}  \end{bmatrix}$ be the balanced multi-splittable BIBD with the given parameters. Let \[H=J_{p^2\times p(p+1)}=J_{p^2\times p(p+1)}-2D=\begin{bmatrix} H_1 & \cdots & H_{p+1}  \end{bmatrix}.\]
Then it is shown that $H$ is a partial Hadamard matrix of order $p^2\times p(p+1)$.  Let $H_{(p+1)/2}=\begin{bmatrix}H_{i_1} & \cdots & H_{i_{(p+1)/2}}\end{bmatrix}$, for any $(p+1)/2$-elements subset $\{i_1,\ldots,i_{(p+1)/2}\}$ of $\{1,2,\ldots,p+1\}$ corresponding to the columns partition in $D$. It is shown that the inner product of any two rows in $H_{(p+1)/2}$ is $\frac{p+1}{2}$ or $-\frac{p+1}{2}$, depending on the inner product of corresponding rows in $D_i$'s is $\frac{p^2-1}{8}$ or $\frac{p^2-2p-3}{8}$ respectively. 
\end{proof}

\section{Equiangular lines}

One of the most critical aspects of balanced multi-splittable partial Hadamard matrices is the existence of a relatively large lower bound for the number of specific equiangular lines in the sense that the inner product of the distinct rows in any of the half of the blocks of the constructed partial Hadamard matrix remains the same in absolute values. 

\begin{lemma}

For an odd prime power $p$ there are $p^2$ equiangular lines in $\mathbb{R}^{\frac{p^2+p}{2}}$ corresponding to $\alpha =\frac{1}{p}$. Furthermore, there are $p^2+p \choose {\frac{p^2+p}{2}}$ choices of block columns of order $p^2\times p$ leading to the same number of equiangular lines.

\end{lemma}

\begin{proof}
A balancedly multi-splittable partial Hadamard matrix of order $p^2\times (p^2+p)$ is constructed in Theorem \ref{thm:main}. The inner product of any two rows of any of the $p^2+p \choose {\frac{p^2+p}{2}}$ choices of block columns is in $\{-\frac{p+1}{2},\frac{p+1}{2}\}$. The corresponding $\alpha = \frac{1}{p}$, and the result follows.
\end{proof}

An asymptotic result concerning the existence of equiangular lines follows. 
\begin{corol}\label{asymp}
The ratio of the number of equiangular lines constructed in Theorem \ref{thm:main} in $\mathbb{R}^{\frac{p^2+p}{2}}$ to the largest possible number of equiangular lines for $\alpha =\frac{1}{p}$ tends to one for large prime powers $p$.
\end{corol}
\begin{proof}
The largest possible number of equiangular lines in $\mathbb{R}^{\frac{p^2+p}{2}}$ corresponding to $\alpha=\frac{1}{p}$ is 
$(p+1)^2$. The required ratio is $(\frac{p+1}{p})^2$, which is close to one for large values of $p$.

\end{proof}

Following the lines of argument in the proof of (i)$\Rightarrow$(ii) in Theorem \ref{thm:main} (page 4), the following is concluded.


\begin{corol}\label{eqalines}
\begin{enumerate}
    \item
    A $(1,-1)$-row vector $r=(r_i)_{i=1}^{p+1}$ of length $p(p+1)$ such that the row vectors of the matrix
    $\begin{bmatrix}
 H_{i_1} & \cdots & H_{i_{(p+1)/2}}\\
 r_{i_1} & \cdots & r_{i_{(p+1)/2}}
\end{bmatrix}$ forms a set of equiangular lines for any $(p+1)/2$-elements subset $\{i_1,\ldots,i_{(p+1)/2}\}$ of $\{1,2,\ldots,p+1\}$ satisfies either
\begin{itemize}
    \item there exists $i$ such that $\pm r_i$ is some row of $H_i$ and $r_j=\mp {\bm 1}$ for any $j\in\{1,\ldots,p+1\}\setminus\{i\}$, or
    \item $r_1=\cdots=r_{p+1}=\pm{\bm 1}$.  
\end{itemize}
\item
There are no two rows vectors $r=(r_i)_{i=1}^{p+1},r'=(r'_{i})_{i=1}^{p+1}$ such that the row vectors of the matrix
    $\begin{bmatrix}
 H_{i_1} & \cdots & H_{i_{(p+1)/2}}\\
 r_{i_1} & \cdots & r_{i_{(p+1)/2}}\\
 r'_{i_1} & \cdots & r'_{i_{(p+1)/2}}
\end{bmatrix}$ forms a set of equiangular lines for any $(p+1)/2$-elements subset $\{i_1,\ldots,i_{(p+1)/2}\}$ of $\{1,2,\ldots,p+1\}$
\end{enumerate}
\end{corol}
 \begin{remark}
 It follows from Corollary \ref{eqalines} that adding more than one line may be possible only for $p=3$. Each of the six choices of nine lines can be extended to the optimal set of sixteen lines in $\mathbb{R}^6$, as shown in the following example. \end{remark}
 \begin{example}
 A balanced multi-splittable $9\times 12$ partial Hadamard matrix:
 \[
\left( \begin{array}{ccc|ccc|ccc|ccc}
 1&1&-  &1&1&- & 1&1&- & 1&1&-\\
  1&1&-  &1&-&1 & 1&-&1 & 1&-&1\\
 1&1&-  &-&1&1 & -&1&1 & -&1&1\\
 1&-&1  &1&1&- & 1&-&1 & -&1&1\\
 1&-&1  &-&1&1 & 1&1&- & 1&-&1\\
 1&-&1  &1&-&1 & -&1&1 & 1&1&-\\
 -&1&1  &1&1&- & -&1&1 & 1&-&1\\
 -&1&1  &1&-&1 & 1&1&- & -&1&1\\
 -&1&1  &-&1&1 & 1&-&1 & 1&1&-
 \end{array}\right)\]
 Any of the possible six selections of two block columns provide nine sets of equiangular lines with the angle $\arccos( \frac{1}{3})$ in $\mathbb{R}^6$. All nine equiangular lines can be extended to sixteen equiangular lines in $\mathbb{R}^6$ by adding the following seven lines.
 \[
 \begin{array}{ccc|ccc}
 - & - & - & 1 & 1 & - \\
  1& 1 & - & - & - & - \\
 - & - & - & 1 & - & 1\\
 1 & - & 1 & - & - & - \\
 - & - & - & - & 1 & 1 \\
 - & 1 & 1 & - & - & - \\
 - & - & - & - & - & - \\
 \end{array} \]
 \end{example}

\section*{Acknowledgments.}
Hadi Kharaghani is supported by the Natural Sciences and
Engineering  Research Council of Canada (NSERC).  Sho Suda is supported by JSPS KAKENHI Grant Number 22K03410.
Yash Khobragade was an MITAC Summer 2023 Research Intern at the University of Lethbridge.



\begin{thebibliography}{99}

\bibitem{BGKM}
A.~E.~Brouwer, C.~D.~Godsil, J.~H.~Koolen, W.~J.~Martin, Width and dual width of subsets in polynomial association schemes, 
{\sl J. Combin. Theory Ser. A}, {\bf 102} (2003), 255--271.

\bibitem{handbook}
C.~J. Colbourn and J.~H. Dinitz, editors.
\newblock \emph{Handbook of combinatorial designs}.
\newblock Discrete Mathematics and its Applications (Boca Raton). Chapman \&
  Hall/CRC, Boca Raton, FL, second edition, 2007.
  
\bibitem{split-od}
Hadi Kharaghani, Thomas Pender, Sho Suda, Balancedly splittable orthogonal designs and equiangular tight frames, Des. Codes Cryptogr., 89 (2021), no. 9, 2033--2050. 

\bibitem{split-h}
Hadi Kharaghani, Sho Suda, Balancedly splittable Hadamard matrices, Discrete Math., 342 (2019), no. 2, 546--561.

\bibitem{KS}
H. Kharaghani,  S. Suda, Hadamard matrices related to projective planes, 
{\sl Electron. J. Comb.} {\bf 30} (2) (2023), P2.49.

\bibitem{stinson}
Douglas Stinson, Combinatorial Designs, Springer-Verlag, New York, 2004.
\end{thebibliography}
\end{document}